\documentclass{amsart}
\usepackage{epsfig}
\usepackage{amssymb}
\usepackage{diagrams}
\usepackage[hypertex]{hyperref}

\usepackage[usenames,dvipsnames]{pstricks}
\usepackage{pst-grad} 
\usepackage{pst-plot} 

\title{An operadic proof of Baez-Dolan Stabilization Hypothesis}
\date{November 16, 2015}

\author{M. A.  Batanin}

\def\End{\mathcal{E}nd}
\def\Ss{\mathcal{S}}

\def\NN{\mathtt{FinSet}}
\def\Ee{\mathcal{E}}
\def\uEe{\underline{\mathcal{E}}}

\def\Set{\mathrm{Set}}
\def\FinSet{\mathrm{FinSet}}
\def\Poly{\mathbf{Poly}}

\def\hocolim{\mathrm{hocolim}}

\def\holim{\mathrm{holim}}

\def\RR{\mathbb{R}}

\newarrow{Into}{littlevee}--->

\theoremstyle{plain}
\newtheorem{theorem}[subsection]{Theorem}

\newtheorem{hypothesis}[subsection]{Hypothesis}
\newtheorem{defin}[subsection]{Definition}
\newtheorem{pro}[subsection]{Proposition}
\newtheorem{lem}[subsection]{Lemma}
\newtheorem{corol}[subsection]{Corollary}
\theoremstyle{remark}
\newtheorem{remark}[subsection]{Remark}
\newtheorem{example}[subsection]{Example}

{\begin{itemize}\tt}%
{\end{itemize}}

\newcommand{\HS}{\mbox{${\bf T}^{\scriptstyle \tt S}$}}

\newcommand{\srk}{\mbox{${\bf SO}_k^{\scriptstyle \tt O(n)}$}}

\newcommand{\sr}{\mbox{${\bf SO}^{\scriptstyle \tt O(n)}$}}

\newcommand{\sri}{\mbox{${\bf SO}^{\scriptstyle \tt O(\infty)}$}}

\newcommand{\ORTrees}{\mbox{${\mathtt{OrderedRootedTrees}}$}}

\newcommand{\nPTrees}{\mbox{${\mathtt{nPlanarRootedTrees}}$}}

\begin{document}

\setcounter{tocdepth}{1}

\begin{abstract} We prove a stabilization theorem for algebras of $n$-operads in a monoidal model category $\Ee.$ It implies a version of  Baez-Dolan stabilization hypothesis for Rezk's weak $n$-categories and some other stabilization results.
\end{abstract}

\maketitle

{\small \tableofcontents}

\maketitle
{\small \tableofcontents}

\section{Introduction}

Breen \cite{breen} and later Baez and Dolan \cite{BD} suggested the following {\it stabilization hypothesis} in higher category theory
\begin{hypothesis} The category of $n$-tuply monoidal $k$-categories is equivalent to the category of   $(n+1)$-tuply monoidal $k$-categories provided $n\ge k+2.$ \end{hypothesis}   

Baez and Dolan  define $n$-tuply monoidal $k$-category as a weak  $n+k$-category  which   
has only  one cell in each dimension smaller  then $n.$  It is known that such a definition, if taken naively, is not completely satisfactory because even if there is a unique cell in lower dimension the action of higher coherence cells  associated to it can be nontrivial (see \cite{Cheng} for a discussion). 

To get rid of this problem we have to work with weakest possible morphisms of $k$-categories.  Moreover, we want to be able to speak about monoidal structures on $k$-categories. One technically convenient way to do it is to choose a
symmetric monoidal model category $(\Ee,\otimes,e)$ whose homotopy category is equivalent to the homotopy category  of weak $k$-categories and weak $k$-functors. For example, for $k=1$ one can take the category of categories with cartesian product and `folklore' model structure and for $k=2$ one can consider the category of $2$-categories with Gray-product as tensor product  and Lack's model structure \cite{Lack}. For any  $k\ge 0$  the  category of $\Theta_k$ simplicial presheaves $\Theta_k Sp_k$  with Rezk model structure satisfies this requirement \cite{BergnerRezk,Rezk}. There is a widely-accepted understanding that a monoidal $k$-category is just an $E_1$-algebra  in such a monoidal model category $\Ee$ \cite{HA}.

Now, an $n$-tuply monoidal weak $k$-category must have $n$ monoidal structures which interact coherently.
 It is another widely-accepted idea  that such  an interaction of   structures is equivalent to an action of an $E_n$-operad \cite{HA}.   This is justified by an additivity theorem for $E_n$-operads (\cite{HA}[Theorem 5.1.1.2]):  the   tensor product of $E_n$-operad and an $E_m$-operad is an $E_{n+m}$-operad. All operads here have to be understood as $\infty$-operads and tensor product  is a `derived' Boardman-Vogt tensor product.  The statement and  proof of this theorem are subtle  because of a  tricky homotopy behaviour of   Boardman -Vogt tensor product  \cite{FV}. 
 Lurie's Additivity  Theorem holds only in fully homotopised world. As a consequence the stabilisation result (\cite{HA}[Example 5.1.2.3]) gives an equivalence of $(\infty,1)$-categories rather than Quillen equivalence of model categories of algebras\footnote{Lurie formulated his argument in the context of $(n,1)$-categories but observed  that it can be extended to a more general context of $(n,k)$-categories. This has been done by Gepner and Haugseng in \cite{GH}.}\footnote{A weaker version of stabilisation hypothesis  was earlier proved by Simpson in \cite{Simpson}.}. 
  
In this paper we choose a  different approach to coherent interaction of  monoidal structures which comes closer to the  original Baez-Dolan understanding of $n$-tuply monoidal $k$-category as a degenerate $(n+k)$-category\footnote{Gepner and Haugseng \cite{GH}  show that such an interpretation is also  possible using their weak enrichment approach.}  and does not require Additivity Theorem (though we conjecture that Additivity Theorem can be proved using our techniques).  {\it An $n$-tuply monoidal $k$-category for us is an algebra of  a cofibrant contractible  $n$-operad in $\Ee.$} 

For $n=1$ this means that a monoidal $k$-category is an algebra of a cofibrant nonsymmetric contractible operad, which is well known to be homotopy equivalent to  an  $E_1$-algebra structure. 
This simple observation was extended in  \cite{SymBat} to arbitrary dimension.  It is shown here  that the derived symmetrisation functor on a terminal $n$-operad is an $E_n$-operad and, therefore, the homotopy category of algebras of cofibrant contractible $n$-operads is equivalent to the homotopy category of $E_n$-algebras. So, homotopically   both approaches   are equivalent.  The difference is that our point of view allows to avoid Boardman-Vogt tensor product and $\infty$-operads.  This has some  advantage as we can use classical operadic and model theoretic methods and  final result is formulated in terms of Quillen equivalences, which is a stronger statement.  Also our techniques allows to prove stabilisation not only for weakly unital $k$-categories but for   nonunital algebras also. It is known that Additivity Theorem fails in this case even for cofibrant operads \cite{FV}[Section 3].  Our proof is essentially the same in all cases, we just need to choose an appropriate category of $n$-operads.

\section{Higher operads and symmetrization}\label{higheroperads}

\subsection{$n$-ordinals and $n$-operads}
Let $n \geq 0$.
Recall~\cite[Sec.~II]{SymBat} that a {\em $n$-ordinal\/} is a 
finite set $T$  equipped with $n$ binary relations  $<_0, \ldots, <_{n-1}$
such that 
\begin{itemize}
\item[(i)] 
 $<_p$ is nonreflexive,
\item[(ii)] 
for every pair $a,b$ of distinct elements of $T$ there
exists exactly one $p$ such that
\[
a<_p b \ \ \mbox{or} \ \ b<_p a,
\]
\item[(iii)] 
if $a<_p b $  and  $b<_q c$  then
$a<_{\min(p,q)} c$.
\end{itemize}
A morphism of $n$-ordinals
$\sigma: T \rightarrow S$ is a map 
of the underlying sets such that  
$i<_p j$ in $T$  implies that
\begin{itemize}
\item[(i)]  
$\sigma(i) <_r \sigma(j)$ for some $r\ge p$, or
\item [(ii)]  
$\sigma(i)= \sigma(j)$, or
\item [(iii)]
$\sigma(j) <_r \sigma(i)$ for $r>p$.
\end{itemize}

Let $Ord(n)$ be the skeletal category of $n$-ordinals and their morphisms. 
Each $n$-ordinal can be represented as a pruned planar rooted tree with $n$ levels (\emph{pruned $n$-tree} for short), cf. \cite[Theorem 2.1]{SymBat}. For example, the $2$-ordinal
\[
0 <_0 2,\ 0 <_0 3,\ 0 <_0 4,\ 1 <_0 2,\  1<_0 3,\ 1 <_0 4,\ 
0 <_1 1,\ 2 <_1 3,\ 2 <_1 4,\ 3 <_1 4, 
\]
is represented by the following pruned tree

\begin{figure}[h]
\[
\psscalebox{.4 .4} 
{
\begin{pspicture}(0,-1.7822068)(6.81,1.7822068)
\psline[linecolor=black, linewidth=0.05](1.06,-0.35779327)(3.06,-1.7577933)(5.06,-0.35779327)(4.96,-0.35779327)
\psline[linecolor=black, linewidth=0.05](0.06,1.0422068)(1.06,-0.35779327)(2.06,1.0422068)(2.06,1.0422068)
\psline[linecolor=black, linewidth=0.05](3.46,1.0422068)(5.06,-0.35779327)(6.66,1.0422068)(6.66,1.0422068)
\psline[linecolor=black, linewidth=0.05](5.06,1.0422068)(5.06,-0.35779327)(5.06,-0.35779327)
\rput[b](0.06,1.5422068){\Huge 0}
\rput[b](2.06,1.5422068){\Huge 1}
\rput[b](3.46,1.5422068){\Huge 2}
\rput[b](5.06,1.5422068){\Huge 3}
\rput[bl](6.66,1.5422068){\Huge 4}
\end{pspicture}
}\]

\caption{\label{symbol}}
\end{figure}

The initial $n$-ordinal $z^nU_0$ has empty underlying set and its representing pruned $n$-tree is degenerate: it has no edges but consists only of the root at level $0$. The terminal $n$-ordinal $U_n$ is represented by a linear tree with $n$ levels.

We also would like to consider the limiting case of $\infty$-ordinals.

\begin{defin}Let $T$ be a finite set equipped with a sequence of  binary antireflexive  complimentary relations  $<_0,<_{-1} \ldots, <_p , <_{p-1} \ldots   $ for all integers $p\le 0.$ The set
  $T$  is called an $\infty$-ordinal if these relations satisfy:
  \begin{itemize}
 \item \  $a<_p b $ \ and \ $b<_q c$ \ implies\
 $a<_{min(p,q)} c .$
 \end{itemize}
 \end{defin}

The definition of morphism between $\infty$-ordinals coincides with the definition of morphism between $n$-ordinals for finite $n.$
The category $Ord(\infty)$ denotes {\it the skeletal category of $\infty$-ordinals .}

For an $n$-ordinal $R$ we consider its  {\it vertical suspension} $S(R)$ which is an $(n+1)$-ordinal with the underlying set $R ,$ and the order $<_m $ equal the  order $<_{m-1}$ on $R$     $<_0$ is empty. 

For example, a vertical suspension of the $2$-ordinal from Figure \ref{symbol} is the $3$-ordinal

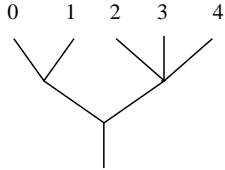
\[
\scalebox{0.4} 
{
\begin{pspicture}(0,-2.795)(7.237695,2.81)
\psline[linewidth=0.05](1.1882031,0.13)(3.188203,-1.2646877)(5.1882033,0.13)(5.088203,0.13)
\psline[linewidth=0.05](0.18820313,1.5353125)(1.1882031,0.13531242)(2.188203,1.5353125)
\psline[linewidth=0.05](3.5882032,1.5353125)(5.1882033,0.13531242)(6.7882032,1.5353125)
\psline[linewidth=0.05](5.18,1.63)(5.18,0.11)
\usefont{T1}{ptm}{m}{n}
\rput(0.1566211,2.4353125){\Huge 0}
\usefont{T1}{ptm}{m}{n}
\rput(2.0829296,2.4353125){\Huge 1}
\usefont{T1}{ptm}{m}{n}
\rput(3.5604882,2.4353125){\Huge 2}
\usefont{T1}{ptm}{m}{n}
\rput(5.1338477,2.4353125){\Huge 3}
\usefont{T1}{ptm}{m}{n}
\rput(6.9860744,2.4353125){\Huge 4}
\psline[linewidth=0.05](3.18,-1.25)(3.18,-2.77)
\end{pspicture} 
}\]

More generally, one can consider a  $p$-suspension $S_p$ where we trivialise the orders $<_p.$ So, the vertical suspension is 
$S = S_0.$

For example, the suspension $S_2$  of the $2$-ordinal from Figure \ref{symbol} is 

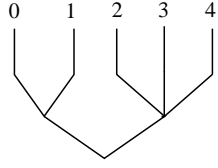
\[ \scalebox{0.4} 
{
\begin{pspicture}(0,-2.6623437)(6.9376955,2.6773438)
\psline[linewidth=0.05](1.1482031,-1.2426562)(3.1482031,-2.637344)(5.148203,-1.2426562)(5.048203,-1.2426562)
\psline[linewidth=0.05](0.14820312,0.1626563)(1.1482031,-1.2373438)(2.1482031,0.1626563)
\psline[linewidth=0.05](3.5482032,0.1626563)(5.148203,-1.2373438)(6.7482033,0.1626563)
\psline[linewidth=0.05](5.14,0.2573438)(5.14,-1.2626562)
\usefont{T1}{ptm}{m}{n}
\rput(0.1566211,2.2826562){\Huge 0}
\usefont{T1}{ptm}{m}{n}
\rput(2.0429296,2.2826562){\Huge 1}
\usefont{T1}{ptm}{m}{n}
\rput(3.5804882,2.3026564){\Huge 2}
\usefont{T1}{ptm}{m}{n}
\rput(5.1138477,2.3026564){\Huge 3}
\usefont{T1}{ptm}{m}{n}
\rput(6.6860743,2.3026564){\Huge 4}
\psline[linewidth=0.05](0.16,1.6773438)(0.16,0.1573438)
\psline[linewidth=0.05](2.14,1.6773438)(2.14,0.1573438)
\psline[linewidth=0.05](3.56,1.6773438)(3.56,0.1573438)
\psline[linewidth=0.05](5.14,1.7573438)(5.14,0.2373438)
\psline[linewidth=0.05](6.74,1.6773438)(6.74,0.1573438)
\end{pspicture} 
}
\]

Suspension operations   give us a family of  functors
$$S_p:Ord(n)\rightarrow Ord(n+1), \ 0\le p \le n.$$ We also define an $\infty$-vertical suspension functor $Ord(n)\rightarrow Ord(\infty)$ as follows. For an $n$-ordinal $T$ its $\infty$-suspension is an $\infty$-ordinal $S^{\infty}T$ whose underlying set is the same as the underlying set of $T$ and $a<_p b$ in $S^{\infty}T$ if
$a<_{n+p-1} b$ in $T .$ It is not hard to see that the sequence
$$ Ord(0)\stackrel{S}{\longrightarrow} Ord(1) \stackrel{S}{\longrightarrow} Ord(2) \longrightarrow \ldots \stackrel{S}{\longrightarrow} Ord(n) \longrightarrow \ldots \stackrel{S^{\infty}}{\longrightarrow} Ord(\infty),$$
exhibits $Ord(\infty)$ as a colimit of $Ord(n) .$\vspace{1ex}

The  categories $Ord(n), 0\le n\le \infty$ are operadic in the sense of Batanin and Markl, cf.  \cite{duoid}. This means that $Ord(n)$ is equipped with 
cardinality  and fiber functors. The cardinality functor  \begin{equation}\label{card}|-|: Ord(n) \to \FinSet\end{equation}  associates to a
$n$-ordinal $T$ its underlying set. Here $\FinSet$ is a skeletal version of the category of finite sets \cite{duoid}.The fiber functor associates to 
each  morphism of $n$-ordinals $\sigma: T \to S$ and  $i \in |S|$ 
the preimage $\sigma^{-1}(i)$ with the induced structure of an
$n$-ordinal. 

The category $\FinSet$ is another example of an operadic category. The fiber functor is given by the preimage like above \cite{duoid}.

Any operadic category $\mathcal{O}$ has an associated category of operads $Op_{\mathcal{O}}(\Ee)$ with values in an arbitrary symmetric monoidal category $(\Ee,\otimes,e)$ \cite{duoid}. 
The category $Op_{\FinSet}(\Ee) = SOp(\Ee)$  is the category of classical symmetric operads in $\Ee.$

The category $Op_n(\Ee)$ of $n$-operads in $\Ee$ is, by definition, the category $Op_{Ord(n)}(\Ee).$

Explicitly an $n$-operad  in $\Ee$ 
is
a collection $\{A_T\}, {T\in Ord(n)}$ of objects in $\Ee$ equipped with the following structure:\vspace{1ex}

- a morphism $\epsilon: e \rightarrow  A(U_n)$ (unit);

- a morphism $m_{\sigma}:A(S)\otimes A(T_0)\otimes \dots \otimes A(T_k)
 \rightarrow A(T)$ (multiplication) for each map of $n$-ordinals $\sigma:T \rightarrow S$, where $T_i = \sigma^{-1}(i), \ i\in\{0,\ldots,k\} = |S|.$\vspace{1ex}

They must satisfy the following identities:\vspace{1ex}

- for any composite map of $n$-ordinals\begin{diagram}[small]T&\rTo^\sigma&S&\rTo^\omega&R\end{diagram}
the associativity diagram

{\unitlength=1mm

\begin{picture}(300,45)(2,0)

\put(20,35){\makebox(0,0){\mbox{$\scriptstyle A(R)\otimes
A(S_{\bullet})\otimes A(T_0^{\bullet}) \otimes  \dots
\otimes
 A(T_i^{\bullet})\otimes  \dots \otimes A(T_k^{\bullet})
$}}}
\put(20,31){\vector(0,-1){12}}

\put(94,31){\vector(0,-1){12}}

\put(90,36.5){\makebox(0,0){\mbox{$\scriptstyle A(R)\otimes
A(S_{0})\otimes A(T_0^{\bullet}) \otimes  \dots
$ }}}

\put(88,33.5){\makebox(0,0){\mbox{$\scriptstyle 
\otimes A(S_{i})\otimes
 A(T_i^{\bullet})\otimes  \dots \otimes A(S_{k})\otimes
A(T_k^{\bullet})
$ }}}

\put(55,36.3){\makebox(0,0){\mbox{$\scriptstyle \cong $}}}
\put(50,35){\vector(1,0){12}}

\put(20,15){\makebox(0,0){\mbox{$\scriptstyle A(S)\otimes
A(T_0^{\bullet}) \otimes  \dots
\otimes
 A(T_i^{\bullet})\otimes  \dots \otimes A(T_k^{\bullet})
$}}}

\put(94,15){\makebox(0,0){\mbox{$\scriptstyle A(R)\otimes
A(T_{\bullet})
$}}}

\put(60,5){\makebox(0,0){\mbox{$ \scriptstyle A(T)
$}}}

\put(35,11){\vector(4,-1){19}}

\put(85,11){\vector(-4,-1){19}}

\end{picture}}

\noindent commutes,
where $$A(S_{\bullet})= A(S_0)\otimes \dots
\otimes A(S_k),$$
$$A(T_{i}^{\bullet}) = A(T_i^0) \otimes \dots\otimes A(T_i^{m_i})$$
and $$ A(T_{\bullet}) =  A(T_0)\otimes \dots
\otimes A(T_k);$$

- for an identity $\sigma = id : T\rightarrow T$ the diagram

{\small \begin{diagram}[small,UO]
A(T)\otimes
{e}\otimes \dots \otimes {e}   &                          &       \rTo               &     &  A(T)\otimes
A(U_n)\otimes \dots \otimes A(U_n)  \\
      \ \ \ \ \ \  \ \ \ \ \ \ \ \ \simeq &\rdTo(2,2)       &        &   \ldTo(2,2) &\quad  \\
          &                         &     A(T)                &      &
\end{diagram}}

\noindent commutes;

- for the unique morphism $T\rightarrow U_n$ the diagram

{\small \begin{diagram}[small,UO]
e \otimes
A(T)
&                          &       \rTo               &     &  A(U_n)\otimes
A(T)
  \\
      \ \ \ \ \simeq &\rdTo(2,2)       &        &   \ldTo(2,2) &\quad  \\
          &                         &     A(T)                &      &
\end{diagram}}

\noindent commutes.

 Functors between operadic categories which preserve cardinalities and fibers are called {\it operadic functors} \cite{duoid}.  The cardinality functor is always an operadic functor. 
 An operadic functor between operadic categories $p:\mathcal{O}\to \mathcal{O}'$ induces a restriction functor
$p^*: Op_{\mathcal{O}'}(\Ee)\to Op_{\mathcal{O}}( \Ee)$ \cite{duoid}. If $\Ee$ is a cocomplete symmetric monoidal category then the restriction functor has a left adjoint $p_!: Op_{\mathcal{O}}(\Ee)\to Op_{\mathcal{O}'}( \Ee).$

Any of the suspension functors is an operadic  functor. In particular the following diagram commutes: 
\begin{gather}\label{cardinality}\begin{diagram}[small,UO]
Ord(n)   &                          &       \rTo^{S_p}               &     &  Ord(n+1)  \\
        \quad |-| &\rdTo(2,2)       &        &   \ldTo(2,2) &|-|\quad  \\
          &                       &     \FinSet                 &      &
\end{diagram}\end{gather}
Hence, it induces the following diagram of adjunctions:

 {\unitlength=1mm
\begin{picture}(200,0)(-10,30)
\put(20,25){\makebox(0,0){\mbox{$Op_n(\Ee)$}}}
\put(24,21){\vector(2,-1){16}}
\put(35,13){\vector(-2,1){16}}

\put(51,13){\vector(2,1){16}}
\put(63,21){\vector(-2,-1){16}}

\put(43,25){\makebox(0,0){\mbox{$ $}}}
\put(43,10){\makebox(0,0){\mbox{$SOp(\Ee)$}}}
\put(58,26){\vector(-1,0){26}}
\put(32,24){\vector(1,0){26}}
\put(40,21){\shortstack{\mbox{$\scriptstyle (S_p)_! $}}}

\put(60,15){\shortstack{\mbox{$\scriptstyle des_{\scriptscriptstyle n+1}^{\scriptscriptstyle } $}}}
\put(43,17){\shortstack{\mbox{$\scriptstyle sym_{\scriptscriptstyle n+1}^{\scriptscriptstyle } $}}}

\put(33,17){\shortstack{\mbox{$\scriptstyle sym_{\scriptscriptstyle n}^{\scriptscriptstyle } $}}}
\put(18,16){\shortstack{\mbox{$\scriptstyle des_{\scriptscriptstyle n}^{\scriptscriptstyle } $}}}

\put(71,25){\makebox(0,0){\mbox{$Op_{n+1}(\Ee)$}}}
\put(41,27.5){\shortstack{\mbox{$\scriptstyle S_p^*$}}}
\end{picture}}

\begin{equation}\label{dessym} \end{equation}

\

\

\

In this diagram the functor $des_n$ is the restriction functor along cardinality functor  
 called {\it desymmetrisation functor} and $sym_n$ is its left adjoint  called {\it symmetrisation functor} \cite{EHBat}.

\begin{example}\rm  Let  $Ass_n\in Op_n(\Ee)$ be an $n$-operad given by   $Ass_{n}(T) = e, T\in Ord(n)(\Ee).$ It is immediately  from the definition of the suspension functors  that 
$S_p^*(Ass_{n+1}) = Ass_n.$ On the other hand, $sym_1(Ass_1) = Ass$ is the classical symmetric operad for monoids while for $n\ge 2 \ $,  $sym_n(Ass_n) = Com$ is the operad for commutative monoids. This is the classical Ekcman-Hilton argument  in disguise, cf. \cite{EHBat}.
\end{example}

Let now $\Ee$ be a closed symmetric monoidal category. An object $X\in \Ee$ has an associated endomorphism symmetric operad $\End_X : $
$$ \End_X(n) = \uEe(X^{\otimes^n},X),$$
where  $\uEe$ is the internal hom of $\Ee.$   

\begin{defin} An algebra of a symmetric operad $A\in SOp(\Ee)$ is an object $X\in \Ee$ equipped with a morphism of operads
$A\to \End_X.$

An algebra of an $n$-operad $B\in Op_n(\Ee)$ is  an object $X\in \Ee$ equipped with a morphism of operads
$B\to des_n(\End_X).$ 
\end{defin}

\begin{lem} \label{algebra} Let $\Ee$ be a cocomplete closed symmetric monoidal category and let $B\in Op_n(\Ee).$  The following categories are equivalent: 
\begin{enumerate}\item
the category of algebras of the $n$-operad $B;$
\item the category of algebras of the $(n+1)$-operad $(S_p)_!(B);$ \item the category of algebras of the  symmetric operad $sym_n(B).$ \end{enumerate}
\end{lem}
\begin{proof} If $\Ee$ is cocomplete the symmetrisation functor $sym_n$ exists and we use the adjuction $sym_n\dashv des_n$ to transform a $B$-algebra structure $B\to des_n(\End_X)$ to a
$sym_n(B)$-algebra structure $sym_n(B)\to \End_X.$ The proof for $(S_p)_!(B)$-algebra structure is similar. 
\end{proof}

\begin{defin} A symmetric operad ($n$-operad) $A\in SOp(\Ee)$ ($A\in Op_(\Ee)$) is called constant-free if $A(0)$ ($A(z^n U_0)$) is an initial object in $\Ee.$ 
\end{defin}

The category $CFSOp(\Ee)$ of constant free symmetric operad is equivalent to the category $Op_{\FinSet_0}(\Ee)$ where $\FinSet_0$ is an operadic subcategory of $\FinSet$ of nonemty finite sets and surjective maps.  Analogously, the category of constant-free $n$-operads $CFOp_n(\Ee)$ is equivalent to $Op_{Ord_0(n)}(\Ee)$ where $Ord_0(n)$ is the operadic category of nonempty $n$-ordinals and surjections. This observation allows us to reformulate all previous  statements for symmetric operads and $n$-operads in the context of constant-free operads. So, the commutative triangle of adjunctions (\ref{dessym}),  as well as the analogue of Lemma \ref{algebra} hold for constant-free operads too.

\subsection{Polynomial monads}

Symmetric and $n$-operads are examples of algebras of polynomial monads. 

\begin{defin} A finitary polynomial $P$ is a diagram in $\Set$ of the form 
\begin{diagram}J&\lTo^s&E&\rTo^p&B&\rTo^t&I\end{diagram}
where $p^{-1}(b)$ is a finite set for any $b\in B.$ \end{defin} 
Each polynomial $P$ generates a functor called {\it polynomial functor} between functor categories
$$\underline{P}:\Set^J \to \Set^I$$
which is defined as the composite functor
\begin{diagram}\Set^J&\rTo^{s^*}&\Set^E&\rTo^{p_*}&\Set^B&\rTo^{t_!}&\Set^I\end{diagram}
So, the functor $\underline{P}$ is given by the formula
\begin{equation}\label{PPP}  \underline{P}(X)_i = \coprod_{b\in t^{-1}(i)} \prod_{e\in p^{-1}(b)} X_{s(e)}, \end{equation}
which explains the name `polynomial' that is  a sum of products of formal variables.

{\it A cartesian morphism between polynomial functors} is their natural transformation such that each naturality square is a pullback.  
Composition of finitary polynomial functors is again a finitary polynomial functor. Finitary polynomial functors and their cartesian morphisms form a $2$-category category $\Poly_f.$ 
\begin{defin}A finitary polynomial monad is a monad in the $2$-category $\Poly_f.$ \end{defin}
\begin{remark} A finitary polynomial functor  preserves filtered colimits and pullbacks. A  polynomial monad is  cartesian  that is its underlying functor preserves pullbacks and its unit and multiplication are cartesian natural transformations. 
\end{remark}
\begin{remark} One can consider more general polynomial functors of nonfinitary type. Since in this paper we don't need these more general functors  we call finitary polynomial monads simply polynomial monads.
\end{remark}

Let $\Ee$ be a cocomplete symmetric monoidal category and $P$ be a polynomial functor. One can construct a functor $\underline{P}^{\mathcal{E}}:\Ee^I \to \Ee^I$ given by the formula similar to (\ref{PPP}):
$$ \underline{P}^{\Ee}(X)_i = \coprod_{b\in t^{-1}(i)} \bigotimes_{e\in p^{-1}(b)} X_{s(e)}.$$
If $I = J$ and $P$ was given a structure of a polynomial monad then $\underline{P}^{\Ee}$ acquires a structure of a monad on $\Ee^I.$
\begin{defin} The category of algebras of a polynomial monad $P$ in a cocomplete symmetric monoidal category $\Ee$ is the category of algebras of the monad $\underline{P}^{\Ee}.$
\end{defin} 

\begin{example}There is a polynomial monad $SO$ whose category of algebras is isomorphic to the category  of symmetric operads. This monad is given by the polynomial 
$$\begin{diagram}\NN&\lTo^s&\ORTrees^*&\rTo^p&\ORTrees&\rTo^t&\NN \end{diagram}$$
in which $\NN$ is the set of isomorphism classes of objects in $\FinSet$ and \\ $\ORTrees$ is the set of isomorphism classes of ordered rooted trees. The multiplication in $SO$ is induced  by insertion of a tree to a vertex of another tree, cf. \cite{BB}[Section 9.4]. 

There is a polynomial monad $O(n)$ whose category of algebras is isomorphic to the category of algebras of $n$-operads.
It is  generated by the polynomial
\begin{diagram}{\tt Ord(n)}&\lTo^s&\nPTrees^*&\rTo^p&\nPTrees_{}&\rTo^t&{\tt Ord(n)}\end{diagram}
where
 $\tt Ord(n)$ is the set of isomorphism classes of $n$-ordinals and $\nPTrees$ is the set of isomorphism classes of $n$-planar trees.
 The multiplication of the monad is induced by insertion of $n$-planar trees into vertices of   $n$-planar trees, cf. \cite{BB}[Proposition 12.15] 
 \end{example}

The commutative triangle (\ref{cardinality}) induces in an obvious way a commutative triangle of polynomial monads
\begin{gather}\label{poltr}\begin{diagram}[small,UO]
O(n)   &                          &       \rTo               &     &  O(n+1)  \\
      |-|  \quad  &\rdTo(2,2)       &        &   \ldTo(2,2) & \quad |-| \\
          &                         &    SO                 &      &
\end{diagram}\end{gather}
\noindent and the triangle of adjunctions (\ref{dessym}) is also  induced by (\ref{poltr}).

\begin{example}There is a polynomial monad $CFSO$ such that its category of algebras is equivalent to the category of constant-free symmetric operads $CFSOp(\Ee).$ \cite{BB}[Section 9.4]. The corresponding generating polynomial is
$$\begin{diagram}\NN_0&\lTo^s&\ORTrees_{reg}^*&\rTo^p&\ORTrees_{reg}&\rTo^t&\NN_0\end{diagram}$$
where $\NN_0$ is the set of isomorphism classes of nonempty finite sets and \\ $\ORTrees_{reg}$ is the set of isomorphism classes of {\it regular} ordered rooted trees. We call a tree regular if for any vertex of the tree the set of incoming edges at this vertex is not empty (so regular trees do not have stumps) . 

Similarly there is a polynomial monad $CFO(n)$ whose category of algebras is equivalent to the category of constant-free $n$-operads $CFOp_n(\Ee),$ cf. \cite{BB}[Proposition 12.19]. It is generated by a polynomial 
\begin{diagram}{\tt ROrd(n)}&\lTo^s&\nPTrees_{reg}^*&\rTo^p&\nPTrees_{reg}&\rTo^t&{\tt ROrd(n)}\end{diagram}
where
 $\tt ROrd(n)$ is the  set of isomorphism classes of nonempty $n$-ordinals; \\ $\nPTrees_{reg}$ is the set of  isomorphism classes of regular $n$-planar trees. 
 \end{example}

 \subsection{Classifiers for maps between polynomial monads}

For any cartesian morphism of cartesian monads $\phi:S\to T$ one can associate a category (in fact a strict categorical $T$-algebra) $\HS$ with certain  universal property \cite{EHBat,BB,W}. This category is called {\it the classifier of internal $S$-algebras inside categorical $T$-algebras}. 

Classifiers allow to compute the left adjoint functor  between categories of algebras induced by $\phi$ is terms of a  colimit over $\HS.$ In particuar,  
the symmetrisation functor $sym_n$ admits an explicit description as a colimit over the classifier $\sr$ of the map of polynomial monads $|-|:O(n)\to SO$. (See \cite{EHBat}.) 

The homotopy type of the nerve of the classifier $\sr$ was computed in \cite{SymBat}:
$$N(\sr)= \coprod_{k}N(\srk)$$ 
where $N(\srk)$ has homotopy type of the configuration space of $k$ points in $\RR^n.$ It follows that for $n\ge 2$ and $0\le i \le n-2$ the homotopy groups
   $\pi_i(N(\srk),a)= 0.$ This can be reformulated as that the nerve of the unique map of categories 
   \begin{equation}\label{!} !: \srk \to 1\end{equation}
   is an $(n-2)$-local weak equivalence of simplicial sets \cite{cis06}[Corollary 9.2.15]. 
   
   For $n=\infty$ we also have a classifier $\sri.$   It is not hard to see that $\sri$ is the colimit of this sequence of  classifiers $\sr$ induced by the vertical suspension functor and so the nerve of  $\sri$  is a contractible simplicial symmetric operad.

\section{Stabilization of algebras of $n$-operads} \label{stabilization}

\subsection{Model categories of symmetric operads, $n$-operads and their algebras}

Now we assume that our  base symmetric monoidal category $(\Ee,\otimes, e)$  is a cofibrantly generated monoidal model category. 
Let $C$ be a set and  $(T,\mu,\epsilon)$ be a monad on the category $\Ee^C.$  There is a product model structure on the category $\Ee^C$ and so one can try to induce a model structure on the category of $T$-algebras as follows. We define an algebra morphism $f:X\to Y$ to be a weak equivalence (fibration) if $U(f)$ is a weak equivalence (resp. fibration) in $\Ee^C,$ where $U$ is the forgetful functor from $T$-algebras to $\Ee^C.$ It is more often with this definition that we will get only a semimodel structure \cite{Fresse,WY} not the full model structure on algebras, but it is sufficient for our purpose. If such a (semi)model structure exists we call it transferred model structure. 
An algebra $X$ of $T$ is called {\it relatively cofibrant}  if $U(X)$ is a cofibrant object in $\Ee^C.$   

\begin{pro}\label{poly} If $e \in \Ee$ is cofibrant then for any polynomial monad $T$ the category $Alg_T(\Ee)$ admits a transferred semimodel structure in which all cofibrant algebras are relatively cofibrant. 
\end{pro}

\begin{proof} The category of algebras of $T$ is isomorphic to the category of algebras of a   coloured symmetric operad $O(T)$ \cite{BB} whose spaces of operations are of the form $e\otimes O(c_1,\ldots,c_m;c) = \sqcup_{O(c_1,\ldots,c_m;c)} e$ where $O(c_1,\ldots,c_m;c)$ is a set with free action of symmetric groups. If $e$ is cofibrant this underlying object of operations  is a $\Sigma$-cofibrant  object and so $O(T)$ is $\Sigma$-cofibrant operad. The statement of proposition follows now from \cite{WY}[Theorem 6.3.1].  
\end{proof}

\begin{pro}\label{homo} If $e$ is cofibrant in $\Ee$ then
\begin{enumerate}  
\item The categories $Op_n(\Ee), CFOp_n(\Ee), SOp(\Ee)$ and $CFSOp(\Ee)$ admit transferred semimodel structures;
 \item Cofibrant symmetric and $n$-operads are relatively cofibrant;
\item The triangle (\ref{dessym}) is a triangle of Quillen adjunctions;
\item  The category of algebras of cofibrant symmetric and cofibrant $n$-operads  admit transferred (semi)model structures; 
\item For any weak equivalence between cofibrant operads $f:A\to B$ the induced  adjunction $f_!\dashv f^*$ between categories of algebras   is a Quillen equivalence.
\end{enumerate}
\end{pro}

\begin{proof} Symmetric operads (general  or constant free) as well as $n$-operads (general  or constant free) are algebras of polynomial monads.
So we are in the conditions of Proposition \ref{poly}. The existence of transferred (semi)model structure on algebras of cofibrant symmetric operads is proven in \cite{Fresse}[Proposition 4.4.3] and \cite{WY}. The existence of transferred model structure on algebras of cofibrant $n$-operads follows from this and Lemma \ref{algebra}. Indeed,  since $sym_n$ is a left Quillen functor $sym_n(A)$ is a cofibrant symmetric operad for any cofibrant $n$-operad $A.$  The last point of the Proposition  is  proven in \cite{Fresse}[Proposition 4.4.6].
\end{proof}

\begin{remark} This semimodel structure on operads is often a full model structures \cite{BB,WY} but not always. For example, the category of symmetric operads ($n$-operads for $n\ge 2$) in the category of chain complexes of finite characteristic does not admit full model structures \cite{BB} but there is a full model structure on the category of constant-free symmetric or $n$-operads for any compactly generated monoidal model category which satisfies the monoid axiom of Schwede and Shipley \cite{BB}.    

\end{remark}

\subsection{Stabilization of algebras}

In this section we prove stabilisation of homotopy categories of algebras of $n$-operads.  The same proof works for constant-free $n$-operads so we do not mention them anymore. 
To simplify notation we fix a $p\ge 0$ and call $p$-suspension of  an $n$-ordinal simply a suspension and we denote it $S:Ord(n)\to Ord(n+1).$ We also denote $S$ the map of polynomial monads induced by the suspension. The proof of our main result does not depend on $p.$

Let $\Ee$ satisfies all assumptions of Proposition \ref{homo}.      Let $G_{n}\in Op_n(\Ee)$ be a cofibrant replacement for $Ass_n .$       We will denote  by $B_{n}(\Ee)$   the category of $G_{n}$-algebras in $\Ee.$ Let also $E_{\infty}(\Ee)$ be the model category of $E_{\infty}$-algebras in $\Ee$ that is the category of algebras of a cofibrant replacement $E$ of the symmetric operad $Com.$ 
\begin{remark}\rm
 The category  $B_{n}(\Ee)$ is equivalent to the category of algebras of the symmetric operad $sym_n(G_{n})$ which is a cofibrant $E_n$-operad, cf. \cite{SymBat}.
\end{remark}

By Lemma \ref{algebra}  there is an isomorphism of categories of algebras of an $n$-operad $G_n$ and an $(n+1)$-operad $S_!(G_{n}).$   
Also observe  that $S_!$ is a left Quillen functor and, hence, preserves cofibrations. In particular, the operad $S_!(G_{n})$ is cofibrant.
There is a  map of $(n+1)$-operads  $i: S_!(G_{n})\to G_{n+1}.$  
Indeed, since $S^*(Ass_{n+1}) = Ass_n$ by adjunction we have  a map $S_!(G_{n})\to Ass_{n+1}.$ We  also have a trivial fibration $G_{n+1}\to Ass_{n+1}.$ Since $S_!(G_{n})$ is cofibrant there is a lifting $i: S_!(G_{n})\to G_{n+1}.$  Without loss of generality we can think that $i$ is a cofibration because if it is not we can always factorise it as  cofibration followed by a trivial fibration and so replace $G_{n+1}$ by another   cofibrant operad with a trivial fibration to $Ass_{n+1}.$   

The morphism  $i$ induces a Quillen adjunction
between algebras of $S_!(G_{n})$ and algebras of $G_{n+1}$ and so between algebras of $G_{n}$ and $G_{n+1}.$ Slightly abusing notations we will denote this adjunction $i^*\vdash i_!.$

Recall that  a
\emph{standard system of simplices} in a monoidal model category $\Ee$ is a cosimplicial object $\delta$ in $\Ee$ satisfying the following properties \cite[Definition A.6]{BergerMoerdijk0}:
\begin{itemize}
\item[(i)]
$\delta$ is cofibrant for the Reedy model structure on
  $\Ee^{\Delta}$,
\item[(ii)] 
$\delta^0$ is the unit object $I$ of $\Ee$ and the
  simplicial operators $[m]\to[n]$ act via weak equivalences
  $\delta^m\to\delta^n$ in $\Ee$, and
\item[(iii)]
the simplicial realization
  functor $|-|_\delta=(-)\otimes_\Delta\delta:\Ee^{\Delta^{op}}\to
  \Ee$ is a symmetric monoidal functor whose structural maps
  \[
|X|_\delta\otimes_V |Y|_\delta\to|X\otimes_V Y|_{\delta}
\]
are weak
  equivalences for Reedy-cofibrant objects $X,Y \in \Ee^{\Delta^{op}}$.\end{itemize}  

Recall also that   a model category $\Ee$  is called {\it $k$-truncated} if for all $X,Y\in \Ee$  $$\pi_i(\widetilde{\Ee}(X,Y), a) = 0 \ ,  \ i > k ,$$   
 for any choice of base point $a.$ Here $\widetilde{\Ee}(X,Y)$ is a homotopy function complex of $\Ee$ \cite{Hirschhorn}.
 
\begin{theorem}\label{stab} Let $(\Ee,\otimes,e)$ be a cofibrantly generated monoidal model category 
whose unit $e\in \Ee$ is cofibrant. 
Then \begin{itemize} \item[(a)] for any $2\le n < \infty$ there is a commutative triangle of Quillen adjunctions:

 {\unitlength=1mm
\begin{picture}(200,30)(-10,5)
\put(20,25){\makebox(0,0){\mbox{$B_{n}(\Ee)$}}}
\put(24,21){\vector(2,-1){16}}
\put(35,13){\vector(-2,1){16}}

\put(51,13){\vector(2,1){16}}
\put(63,21){\vector(-2,-1){16}}

\put(43,25){\makebox(0,0){\mbox{$ $}}}
\put(43,10){\makebox(0,0){\mbox{$E_{\infty}(\Ee)$}}}
\put(58,26){\vector(-1,0){26}}
\put(32,24){\vector(1,0){26}}
\put(40,21){\shortstack{\mbox{$i_! $}}}

\put(38,17){\shortstack{\mbox{$ $}}}


\put(68,25){\makebox(0,0){\mbox{$B_{n+1}(\Ee)$}}}
\put(40,27){\shortstack{\mbox{$i^* $}}}
\end{picture}}
\item[(b)]  If $\Ee$ has a standard system of simplices      then there is a Quillen equivalence    

 {\unitlength=1mm
\begin{picture}(200,15)(-10,17)
\put(22,25){\makebox(0,0){\mbox{$B_{\infty}(\Ee)$}}}


\put(58,26){\vector(-1,0){26}}
\put(32,24){\vector(1,0){26}}

\put(38,17){\shortstack{\mbox{$ $}}}


\put(68,25){\makebox(0,0){\mbox{$E_{\infty}(\Ee)$}}}
\end{picture}}

\item[(c)]  If, in addition, $\Ee$ is $k$-truncated then  the triangle from (a)
 is a  triangle of Quillen equivalences for any $ n \ge k+2 .$ 
\end{itemize}

\end{theorem}

\begin{proof}  Apply the symmetrisation functor $sym_{n}$ to the cofibrant replacement  $G_{n}\to Ass_{n}.$  
We have a morphism $P_{n}: sym_{n}(G_{n})\to sym_{n}(Ass_{n})= Com$ and, hence a lifting 
of this morphism to the morphism of operads  $sym_{n}(G_{n})\to E.$ By (\ref{dessym}) we can replace it by a morphism
$sym_{n+1}(S_!(G_{n}))\to E.$ 
Applying $sym_{n+1}$ to the cofibration $i: S_!(G_{n})\to G_{n+1}$ we have a composite 
$$sym_{n+1}(S_!(G_n))\to sym_{n+1}(G_{n+1}) \to sym_{n+1}(1_{n+1}) = Com$$
and since $G_n$ is cofibrant we have a lifting $$sym_{n+1}(S_!(G_n))\to E.$$
So, we have a commutative   
 square of operads

{\unitlength=1mm

\begin{picture}(40,34)(-29,0)
\put(13,25){\makebox(0,0){\mbox{$sym_{n+1}(S_!(G_{n})) $}}}


\put(21.8,11){\line(3,2){4}}
\put(27.5,14.8){\line(3,2){4}}
\put(34,19){\vector(3,2){4}}

\put(12,21){\vector(0,-1){10}}%
\put(28,25){\vector(1,0){11}}
\put(29,8){\shortstack{\mbox{$\scriptstyle P_{n+1}$}}}
\put(-5,15){\shortstack{\mbox{$\scriptstyle sym_{n+1}(i)$}}}
\put(42,25){\makebox(0,0){\mbox{$E$}}}
\put(42,21){\vector(0,-1){10}}
\put(13,7){\makebox(0,0){\mbox{$sym_{n+1}(G_{n+1})$}}}
\put(27,7){\vector(1,0){11}}
\put(44,7){\makebox(0,0){\mbox{$Com$}}}
\end{picture}}

 \noindent and, hence, a lifting $sym_{n+1}(G_{n+1})\to E.$ The upper commutative triangle of operads induces the triangle of Quillen adjunctions. This proves  statement (a).
 
 Let us first prove statement (c) of the theorem, so
we assume that $\Ee$ is $k$-truncated. 
 
 By construction  the composite of left vertical morphism and the bottom horizontal morphism is $sym_{n+1}(S_!(P_n))$   and  by naturality of isomorphism $sym_n \simeq sym_{n+1}(S_!)$  is isomorphic to $P_n.$  To finish the proof it will be enough to show that $P_{n}$ and $P_{n+1}$ are  weak equivalences of operads and so $sym_n(G_{n})$ and $sym_{n+1}(G_{n+1})$ are both cofibrant replacements of $Com$ in the category of symmetric operads in $\Ee.$  The morphism  $sym_{n+1}(i)$ is then a weak equivalence by two out of three property.  
 
 Since $G_{n}$ is cofibrant the operad $sym_n(G_{n})$ is weakly equvalent to the operad $\mathbb{L}sym_n(G_{n}),$ where $\mathbb{L}sym_n$ is the left derived  symmetrisation functor. The underlying object of $G_{n}$ is cofibrant and $\Ee$ has standard system of simplices so we can apply Theorem  8.2 from \cite{BB}. This theorem states  that 
 $\mathbb{L}sym_n(G_{n})(T)$ is the homotopy colimit in $\Ee$ of a diagram  $\widetilde{G_{n}}: \sr \to \Ee.$

  The
functor $\widetilde{G_{n}}$ representing 
the $n$-operad $G_{n}$ has value on an object $\tau \in  \sr$ given by a certain tensor product of values of the operad $G_{n}$ and, hence, the functor 
$\widetilde{G_{n}}$ is equipped with a canonical  weak equivalence   $\widetilde{G_{n}}(\tau)\to !^*(e),$ where $!^*(e)$ is the constant functor on $\sr$ whose value is the tensor unit $e.$ 
Since both functors $\widetilde{G_{n}}$ and  $!^*(e)$  are pointwise cofibrant we have a  weak equivalence of homotopy colimits.
It remains to show that the canonical morphism
$$\hocolim_{{\bf SO}^{\tt O(n)}}!^*(e)\to e$$
is a weak equivalence. For this it is enough to prove that for any fibrant object $S\in \Ee$ the induced map of simplicial sets
 $$\widetilde{\Ee}(\hocolim_{{\bf SO}^{\tt O(n)}}!^*(e), S)\leftarrow \widetilde{\Ee}(e,S)$$ 
 is a weak equivalence.
Equivalently, we have to prove that for any fibrant $k$-truncated simplicial set $W$ the map
\begin{equation}\label{holim}\holim_{{\bf SO}^{\tt O(n)}}!^*(W) \leftarrow W\end{equation}
is a weak equivalence. Let $\Ss(-,-)$ be the internal hom in simplicial sets.  We have 
$$\Ss(N(\sr),W) \simeq \Ss(\hocolim_{{\bf SO}^{\tt O(n)}} !^*(1) , W) \simeq $$ $$ \simeq\holim_{{\bf SO}^{\tt O(n)}}!^*(\Ss(1,W)) = \holim_{{\bf SO}^{\tt O(n)}} !^*(W),$$
and the map (\ref{holim}) is induced by  (\ref{!})
so it is a weak equivalence since $N(!)$ is an $(n-2)$-equivalence and,  hence, $i$-equivalence for each $i\le n-2.$   So, we proved point (c) of the Theorem. 

The argument for (b) is identical but we don't need $\Ee$ to be truncated because the classifier of $\infty$-operads inside symmetric operads is contractible.

\end{proof}

\begin{corol}[Stabilisation for weak $k$-groupoids] \label{wkg} The  suspension functor induces an  equivalence between homotopy category of
 $n$-tuply monoidal weak $k$-groupoids  and $(n+1)$-tuply monoidal weak $k$-groupoids for $n\ge k+2.$ \end{corol}

\begin{proof} We apply Theorem \ref{stab} to the category of homotopy $k$-types $Sp_k$  which is the $k$-truncation of the model category of simplicial sets $Sp =Set^{\Delta^{op}}$ with its Kan model structure \cite{cis06}.  Weak $k$-groupoids are fibrant objects in this category.    
\end{proof}

\begin{remark}\rm Corollary \ref{wkg} implies classical Freudental stabilisation theorem (cf. \cite{BD}). 
\end{remark}

Recall that Rezk's $(m+k,m)$-categories are fibrant objects in the model category $\Theta_m Sp_k, -2\le k\le \infty $ which is a 
truncation of the model category of Rezk's complete  $\Theta_m$-spaces $\Theta_m Sp_{\infty}$ .  The category $\Theta_m Sp_{\infty}$ is itself a certain Bousfireld localisation of the category of simplicial presheaves $Sp^{\Theta_m^{op}}$ with its  injective model structure. 
This is a cartesian closed model category which is $(m+k)$-truncated and satisfies all hypothesis of Theorem \ref{stab}(see \cite{Rezk}).

\begin{defin} The  category of  Rezk's $n$-tuply monoidal $(m+k,m)$-categories is the category  of fibrant objects in the (semi)model category $B_n(\Theta_m Sp_k).$ 
\end{defin}

We immediately have

\begin{corol}[Stabilisation for Rezk's $(m+k,m)$-categories]\label{wkc} The  suspension functor induces an  equivalence between homotopy category of
Rezk's $n$-tuply monoidal $(m+k,m)$-categories  and Rezk's $(n+1)$-tuply monoidal $(m+k,m)$-categories for $n\ge m+k+2.$  
\end{corol}

\begin{remark} If $m=0$ the category  $\Theta_0 Sp_k$ is isomorphic (as a cartesian model category) to the category $Sp_k$ (cf. \cite{Rezk}) and so  the Corollary \ref{wkc} is a particular case of Corollary \ref{wkg}.

If $k=0$ the fibrant objects of the category $\Theta_m Sp_m$ are weak $m$-categories and so we proved classical Baez-Dolan Stabilization Hypothesis for Rezk $m$-categories.

\end{remark}

\begin{remark} \rm 

The choice of the suspension functor amounts to the choice of a  multiplicative structure on an algebras from $B_{n+1}(\Ee)$ which we would like to `forget'.  Theorem \ref{stab} asserts that up to homotopy this choice in stable dimensions is not important. \end{remark}

\begin{remark} The argument of the Theorem \ref{stab} works equally well for the Swiss-Cheese type symmetric and $n$-operads \cite{SymBat}. The stabilisation result amounts then to the stable version of the Swiss-Cheese conjecture of Kontsevich, cf. \cite{Kon}.\end{remark}

\begin{remark}
Another conclusion from the proof of the Theorem \ref{stab} is that some interesting results about equivalence of homotopy categories of algebras can be proved  once we have a map of polynomial monads $\phi:S\to T$ such that  the classifier $\HS$ is aspherical  with respect to a fixed fundamental localiser $\mathcal W$ \cite{cis06}. We hope to make use of this observation in a future.\end{remark}

\noindent {\bf Acknowledgements.}    I  wish to express my  gratitude to
 C. Berger, D.-C.Cisinski, E. Getzler, R. Haugseng, A. Joyal, S. Lack, M. Markl, R. Street,  M. Weber, D. White for many useful discussions. I am also grateful to the referee for many suggestions which allow to improve the presentation of the result.

The  author also  gratefully acknowledges  the financial
support of Scott Russel Johnson Memorial Foundation, Max Planck
Institut f\"{u}r Mathematik and Australian Research Council (grants
No.~DP0558372, No.~DP1095346). 

\renewcommand{\refname}{Bibliography.}

\

\

\noindent{\small\sc Macquarie University, \\
North Ryde, 2109 Sydney, Australia.} \\  \hspace{2em}\emph{E-mail:}
michael.batanin$@$mq.edu.au\vspace{1ex}

\

\end{document}